\documentclass[a4paper,twoside]{article}

\usepackage[pdftex]{hyperref}  %for pdflatex
\usepackage{amsrefs}

\usepackage{amsmath,amsthm,amsfonts,latexsym,amscd,amssymb,enumerate}

\swapnumbers
\theoremstyle{plain}
\newtheorem{theorem}{Theorem}[section]
\newtheorem{lemma}[theorem]{Lemma}
\newtheorem{corollary}[theorem]{Corollary}

\theoremstyle{definition}

\theoremstyle{remark}
\newtheorem{remark}[theorem]{Remark}

\DeclareMathOperator{\Pos}{Pos} 
\DeclareMathOperator{\spin}{spin}

\newcommand{\reals}{\mathbb{R}}
\newcommand{\complexs}{\mathbb{C}}

\newcommand{\integers}{\mathbb{Z}}

\newcommand{\abs}[1]{\left\lvert#1\right\rvert} %absolute value

\newcommand{\tensor}{\otimes}
\newcommand{\into}{\hookrightarrow}

\newcommand{\iso}{\cong}

\DeclareMathOperator{\im}{im}      %image
\DeclareMathOperator{\tr}{tr}

\DeclareMathOperator{\ind}{Ind}
\DeclareMathOperator{\Ind}{Ind}

\newcommand{\forget}[1]{}

{\catcode`@=11\global\let\c@equation=\c@theorem}

\allowdisplaybreaks[2]

\begin{document}
\pagestyle{myheadings}
\markboth{P.~Piazza, T.~Schick, V.~F.~Zenobi }{On positive scalar
  curvature bordism}

\title{On positive scalar curvature bordism}

\author{ Paolo Piazza, Thomas Schick, Vito Felice
  Zenobi\footnote{PP thanks Ministero Istruzione Universit\`a Ricerca for partial
support through the PRIN 2015 {\em Spazi di Moduli e Teoria di
Lie}. TS and VFZ thank the German Science Foundation and its
    priority program ``Geometry at Infinity'' for partial support.}}
\date{}
% \author{ Thomas Schick\thanks{
% \protect\href{mailto:thomas.schick@math.uni-goettingen.de}{e-mail:
%   thomas.schick@math.uni-goettingen.de}  
% \protect\\
% \protect\href{http://www.uni-math.gwdg.de/schick}{www:~http://www.uni-math.gwdg.de/schick}}\\
% Mathematisches Institut\\
% Universit\"at G{\"o}ttingen\\
% Germany }
\maketitle

\begin{abstract}
  Using standard results from higher (secondary) index theory, we prove that
  the positive scalar curvature bordism groups
  $\Pos_{4n}^{\spin}(G\times\integers)$ are infinite for any
  $n\ge 1$ and $G$ a group with non-trivial torsion. We construct
  representatives of each of these classes which are connected and with
  fundamental group $G\times \integers$. We get the same result for
  $\Pos^{\spin}_{4n+2}(G\times\integers)$ if $G$ is finite and contains an element
  which is not conjugate to its inverse. This generalizes
  the main result of Kazaras, Ruberman, Saveliev, ``On positive scalar
  curvature cobordism and the conformal Laplacian on end-periodic manifolds''
  to arbitrary even dimensions and arbitrary groups with torsion.
\end{abstract}

\section{Introduction}

The classification of Riemannian metrics of positive scalar curvature (up to
suitable natural equivalence relations) is an active object of study in
geometry. One popular way to organize this uses the \emph{Stolz positive scalar
curvature exact sequence} (compare \cite[Proposition 1.27]{PiazzaSchick_Stolz}
  \begin{equation}\label{eq:Stolz_seq}
   \cdots \to \Pos^{\spin}_n(B\Gamma)\to \Omega_n^{\spin}(B\Gamma)\to R^{\spin}_{n}(B\Gamma) \to
    \Pos^{\spin}_{n-1}(B\Gamma)\to\cdots
  \end{equation}
Here, $\Pos^{\spin}_n(B\Gamma)$ is one of our main objects of study, the group of
closed $n$-dimensional spin manifolds $(M,f,g)$ with a reference map $f\colon
M\to B\Gamma$ and a Riemannian metric $g$ of positive scalar
curvature. The equivalence relation is bordism, where the bordisms have to
carry the corresponding structure. We use throughout the usual convention
that Riemannian metrics on manifolds with boundary (e.g.~on a bordism) must have product structure near the boundary. 

In this setup, $\Gamma$ is just an
arbitrary group. If the starting point is a connected smooth manifold $M$, typically one
chooses $\Gamma=\pi_1(M)$. Moreover, $B\Gamma$ is a classifying space for
$\Gamma$ and the map $f$ contains essentially the same information as the
homomorphism $f_*\colon \pi_1(M)\to \Gamma$.

The group $\Omega^{\spin}(B\Gamma)$ is the usual
bordism group, whereas $R^{\spin}_n(B\Gamma)$ is Stolz' $R$-group, defined as the set of bordism classes $(W,f,g)$ where $W$ is a
    compact $(n+1)$-dimensional spin-manifold, possibly with boundary, with a
    reference map 
    $f\colon W\to B\Gamma$, and with a positive scalar curvature metric \emph{on the
      boundary} when the latter is non-empty.

    Our main goal is to show that the positive scalar curvature bordism groups
    $\Pos^{\spin}_n(B\Gamma)$ are rich (more precisely, map onto an infinite
    cyclic group) in new situations. Because the starting point often is a
    fixed connected manifold $M$ with fundamental group $\Gamma$, we will show
    in addition that the infinitely many different non-trivial representatives
    can be chosen to be connected and with fundamental group $\Gamma$ (mapped
    bijectively under the reference map to $B\Gamma$), even in dimension
    $4$. %  In higher dimensions, one can even choose representatives with the
    % fixed connected manifold $M$ with $\pi_1(M)=\Gamma$.
Indeed, we are taking up the main result \cite[Theorem 1]{KazarasRubermanSaveliev} of Kazaras, Ruberman, and Saveliev, which says
\begin{theorem}\label{theo:KRS}
  Let $n=4$ or $6$ and
  $\{1\}\ne G$ be a fundamental group of a $3$-dimensional spherical space
  form if
  $n=4$, or more generally a finite group with at least one element not
  conjugate to its inverse if $n=6$. Set $\Gamma:=G\times\integers$.

  Then $\Pos_n^{\spin}(B\Gamma)$ contains infinitely many elements,
  represented by maps $f\colon M\to B\Gamma$ where $M$ is connected and $f$
  induces an isomorphism in $\pi_1$.
\end{theorem}

We improve that theorem by allowing more general groups and all even
dimensions bigger than $2$.
\begin{theorem}\label{theo:main}
  Let $n>2$ be an even integer. If $n=4k$ is divisible by $4$, let $G$  be an
  arbitrary finitely presented group which contains a non-trivial torsion
  element. If $n=4k+2$ 
  let $G$ be a finite group such that at least one of its elements is not
  conjugate to its inverse. Set $\Gamma:=G\times\integers$.

  Then $\Pos_n^{\spin}(B\Gamma)$ contains infinitely many elements $x_j$,
  represented by connected manifolds $M_j$ with fundamental group $\Gamma$ as in the
  Theorem of Kazaras, Ruberman, and Saveliev.

  Even better, we have a non-trivial homomorphism $\Ind_\rho\colon R^{\spin}_{n+1}(B\Gamma)\to
  \reals$ and the $x_j$ lift to elements in $R^{\spin}_{n+1}(B\Gamma)$ whose
  image under $\Ind_\rho$ form an infinite cyclic subgroup of $\reals$.
\end{theorem}

Theorem \ref{theo:KRS} of \cite{KazarasRubermanSaveliev} is based on the beautiful, but
rather complicated and technical index theory for manifolds with periodic
ends. The associated bordism invariance is very delicate and uses, for the
time being, minimal hypersurface techniques which cause the dimension
restriction in Theorem \ref{theo:KRS}

The main point we want to make in this note is that this delicate theory is
actually not necessary for the result at hand. Instead, it can be derived from
well known (and not too complicated) results in higher index theory, and along
the way one gets the more general result. Related techniques and results are
developed in \cite{XieYuZeidler} which can be used to get even stronger results,
compare the discussion below.

\section{Secondary higher index}

The distinction of bordism classes of metrics of positive scalar curvature in
the literature typically relies on secondary index invariants of the Dirac
operator, and this is 
precisely how we prove the main part of Theorem \ref{theo:main}.

Recall that the classical (higher) index of the Dirac operator can be defined
for closed manifolds, but also for manifolds with boundary provided the
boundary operator is invertible \cite[Section 2.2]{PiazzaSchick_Stolz}. In particular, we have a commutative diagram
\begin{equation}\label{eq:Gamma}
  \begin{CD}
    \Omega^{\spin}_n(B\Gamma) @>>> R^{\spin}_n(B\Gamma)\\
    @VVV @VV{\ind}V\\
    K_n(B\Gamma) @>{\mu}>> K_n(C^*\Gamma)
  \end{CD}
\end{equation}
Here, $\mu$ is the Baum-Connes assembly map which in this setup is \emph{not}
an isomorphism if $\Gamma$ is not torsion-free.

This diagram is the easy and elementary part of the diagram ``mapping the Stolz exact
sequence to analysis'' developed in \cite{PiazzaSchick_Stolz}. Recall that $\Gamma=G\times \integers$. The diagram
\eqref{eq:Gamma} and the corresponding one for $G$ are closely related via a
kind of K\"unneth theorem. Indeed, we have a transformation, given by product
with $S^1$ or its fundamental K-theory class (and using that $B\Gamma=BG\times
S^1$ with $S^1=B\integers$)
\begin{equation}
  \label{eq:from_G_to_Gamma}
  \begin{CD}
    \Omega^{\spin}_{n}(BG)  @>>> R^{\spin}_{n}(BG) &&&&  \Omega^{\spin}_{n+1}(B\Gamma) @>>> R^{\spin}_{n+1}(B\Gamma)\\
    @VVV @VV{\ind}V  @>{\times [S^1]}>>  @VVV @VV{\ind}V\\
    K_{n}(BG)  @>{\mu}>> K_{n}(C^*G) &&&& K_{n+1}(B\Gamma) @>{\mu}>> K_{n+1}(C^*\Gamma)
  \end{CD}
\end{equation}

\begin{lemma}
  After removing the term $R^{\spin}_*(\cdot)$, the map $\times[S^1]$ of  \eqref{eq:from_G_to_Gamma} is an
  embedding as direct summand of the left diagram into the right. In other
  words, the maps $\times [S^1]$ are injective and the map $\mu$
  is compatible with a suitable direct sum decomposition of the right hand
  side into the left hand side and a complement.
\end{lemma}
\begin{proof}
  This is certainly well known. For the convenience of the reader we give
  here a detailed proof.
  
  For a multiplicative generalized homology theory $E_*$ and for natural
  transformations of such, in particular for $\Omega^{\spin}_*\to K_*$, this
  is a 
  consequence of the natural K\"unneth theorem for the special product
  $X\times S^1$ 
  or equivalently of the suspension isomorphism. The desired splitting
  of $E_*(X)\xrightarrow{\times [S^1]} E_{*+1}(X\times S^1)$ is obtained as
  the composition
  \begin{equation*}
  E_{*+1}(X\times S^1)\xrightarrow{pr} \tilde E_{*+1}(X\times S^1/X\times *)\xrightarrow{=}
  \tilde E_{*+1}(\Sigma (X_+))\xrightarrow[\iso]{\sigma^{-1}}\tilde E_*(X_+)=E_*(X).
\end{equation*}
Here $pr$ is the canonical projection, $X_+$ is the disjoint union of $X$ and
an additional basepoint, $\Sigma (X_+)$ its reduced suspension, and $\sigma$ is
the suspension
 isomorphism. The complementary summand is the (injective) image
 $i_*(E_{*+1}(X))\subset E_{*+1}(X\times S^1)$ for $i\colon X\into X\times
 S^1; x\mapsto (x,*)$ (or the exterior product with the unit $1$), so that we
 get in particular the direct sum decomposition 
 \begin{equation*}
 K_{*+1}(X\times S^1)= K_{*+1}(X)\otimes 1 \oplus K_*(X)\otimes [S^1].
\end{equation*}
For the $C^*$-algebra K-theory, the K\"unneth theorem \cite[Theorem
4.1]{Schochet} 
 specializes in our situation to the isomorphism (given by external tensor
 product) $K_*(C^*G)\tensor
 K_*(C^*\integers)\to K_*(C^*G\tensor C^*\integers)=K_*(C^*\Gamma)$. Using
 $K_*(C^*\integers) =1\cdot \integers\oplus [S^1]\cdot\integers$ with unit
 $1\in K_0$ and $[S^1]\in K_1(C^*\integers)$, the image of $[S^1]$ under the Baum-Connes
 isomorphism $\mu_{S^1} \colon K_*(S^1)\xrightarrow{\iso} K_*(C^*\integers)$, we
 obtain the splitting
 \begin{equation*}
   K_{*+1}(C^*\Gamma) = K_{*+1}(C^*G)\tensor 1 \oplus K_*(C^*\Gamma)\tensor
   [S^1].
 \end{equation*}
  The Baum-Connes map $\mu$ is compatible with external products, proving the
  rest of the claims.
\end{proof}

% Note that the K\"unneth formula for K-homology and the K-theory of
% $C^*$-algebras says that, in the lower row, this map is an embedding of a
% direct summand.

The strategy is now the following.
\begin{enumerate}
\item\label{item:deloc} We recall a suitable homomorphism $\ind_\rho\colon K_{n}(C^*G)\to \reals$, a
  delocalized index. ``Delocalized'' means in particular that $\ind_\rho\circ
  \mu\colon  K_{n}(BG)\to \reals$ is zero.
\item We will construct an appropriate element $[W,f,g]\in R^{\spin}_{n}(BG)$ with
  $\ind_\rho(\ind[W,f,g])\ne 0$. It generates an infinite cyclic subgroup of
  $R_{n}^{\spin}(BG)$ detected in $K_{n}(C^*G)$ and then in $\reals$ as image of
  the maps $\ind$ and $\ind_\rho\circ \ind$, respectively.
\item By the commutativity of \eqref{eq:Gamma} and using item \ref{item:deloc}, no
  non-trivial element of this infinite cyclic subgroup is in the image of
  $\Omega_{n}^{\spin}(BG)\to R^{\spin}_{n}(BG)$, and therefore this
  infinite cyclic 
  subgroup is mapped injectively to $\Pos_{n-1}^{\spin}(BG)$.
\end{enumerate}

 Finally, we take the Cartesian product with $S^1$.
\begin{corollary}
 The element $W\times S^1$
  generates an infinite cyclic subgroup of $R^{\spin}_{n+1}(B\Gamma)$ detected
  in $K_{n+1}(C^*\Gamma)$ as image of the map $\ind$. 

  None of its non-zero
  elements is in the image of $\Omega^{\spin}_{n+1}(B\Gamma)$, and so this
  infinite cyclic 
  group generated by $W\times S^1$ injects into $\Pos^{\spin}_{n}(B\Gamma)$. 
\end{corollary}
\begin{proof}
  This follows from the
  commutativity of \eqref{eq:Gamma} and the fact that the K\"unneth map
  \eqref{eq:from_G_to_Gamma} gives an inclusion as direct summand of the whole map
  $ K_n(BG)\xrightarrow{\mu} K_n(C^*G)$ into
  $K_{n+1}(B\Gamma)\xrightarrow{\mu}K_{n+1}(C^*\Gamma)$.
\end{proof}

It remains to construct $W$ and $\ind_\rho$ with the appropriate
properties. In all cases, $W$ will be a null bordism of a disjoint union of
lens spaces or lens space bundles, we describe this explicitly later.

\subsubsection{Case $n\equiv 0\pmod 4$}

Essentially, this case has been
treated in \cite{PiazzaSchick} and we refer to this paper for more details. We choose $\ind_\rho\colon K_0(C^*G)\to \reals$ to be equal
to the difference of the homomorphisms induced by the standard trace and
by the trivial representation. Recall that the first one gives the $L^2$-index and it is induced by the
canonical trace $\tr_G\colon C^*G\to\complexs$, while the second one gives the ordinary index and  is induced by the
trivial homomorphism $G\to \{1\}$.

It is a direct consequence of Atiyah's $L^2$-index theorem \cite[(1.1)]{Atiyah} that $\ind_{\rho}$
defined in this way vanishes on the image of $\mu$.

Small caveat: we have to use a
$C^*$-completion $C^*G$ of $\complexs[G]$ such that the trivial homomorphism
extends to this completion. For this, one could use the maximal
$C^*$-completion or a smaller, more geometric one (based on the direct sum of
the regular representation and the trivial representation). The relevant
theory is well known, compare e.g.~\cite{SchickSeyedhosseini}.

We have to compute $\ind_\rho$ for a compact manifold $W$ of dimension
$n$ with boundary $L$ of dimension $n-1$,
where the boundary is equipped with a metric of positive scalar curvature, and
where $W$ is equipped with a reference map to $BG$. Again, this is a classical
result, an easy special case of the ($L^2$)-Atiyah-Patodi-Singer index
theorem \cite{Ramachandran}: this index is the Cheeger-Gromov $L^2$-$\rho$-invariant of the
positive scalar
curvature metric of the boundary $L$, where we
use the $G$-covering pulled back via the map to $BG$.

Finally, for a cyclic group $C_p$ of order $p$ which we choose to be prime,
the $L^2$-$\rho$-invariant is
just the Atiyah-Patodi-Singer $\rho$-invariant associated to
$\alpha:=\frac{1}{p}\rho_{reg}-\rho_1$, the linear combination of the trivial
and the regular representation in the complexified representation ring.

For the group $G$ with non-trivial torsion element, we choose an embedding
$\iota\colon C_p\to G$ with induced map $B\iota\colon BC_p\to BG$. The
$L^2$-$\rho$-invariant is compatible with ``induction'', meaning that for a
map $L\to BC_p\xrightarrow{B\iota}BG$ the $L^2$-$\rho$-invariants with respect
to the covering pulled back from $BC_p$ and from $BG$ coincide.
In this situation, therefore, $\rho_{(2)}(L)=\rho_{\alpha}(L)$, see \cite[Lemma 2.22]{PiazzaSchick}.

Finally, we come to the explicit construction of $W$. We choose for the
boundary $L$ the disjoint union of an appropriate number $N_L$
of copies of the lens space $L(p;1,\dots,1)$, the quotient of
$S^{n-1}\subset \complexs^{n/2}$ by the action of $C_p$ where the generator
acts by multiplication with a fixed $p$-th root of unity, equipped with the
quotient metric $g_L$ and with map $u\colon L(p;1,\dots,1)\to BC_p$ inducing an isomorphism on the
fundamental group. This space has a spin structure (unique if $p>2$). The 
computation of the $\eta$-invariant of the Dirac operator is classical, we
get:
\begin{lemma}
  \begin{equation*}
    \rho_{(2)}(L(p,1,\dots,1)) = \rho_{\alpha}(L) =
    (-1)^{\frac{n}{4}+1}\frac{1}{p}\sum_{j=1}^{p-1}\frac{1}{\abs{\zeta^j-1}^{n/2}} \ne 0.
  \end{equation*}
  Here, $\zeta$ is a primitive $p$-th root of unity. 
\end{lemma}
\begin{proof}
  This follows from the equivariant Lefschetz fixed point formula of Donnelly
  \cite[Proposition 4.1]{Donnelly} (applied to $S^{n-1}$ bounding an $n$-dimensional
  hemisphere), as observed in \cite[Lemma 2.3]{BotvinnikGilkey_eta}, compare
  also \cite[Lemma 2.1]{Gilkey} or \cite{HansonRoemer}. Note that the formulas
  hold in $\reals$ and not only in $\reals/\integers$
  because we have positive scalar curvature and hence no harmonic spinors. We
  use that for our virtual representation $\rho$ we have $\rho(g)=-1$ if
  $g\ne e$ and $\rho(e)=0$.
\end{proof}

The group $\Omega^{\spin}_{n-1}(BC_p)$ is finite because $n-1$ is odd, as one
deduces readily from the Atiyah-Hirzebruch spectral sequence. Therefore, we
find $N_L$ and a spin null bordism $U\colon W\to BC_p$ of $L:= N_L\cdot
L(p;1,\dots,1)= \bigsqcup_{i=1}^{N_L} L(p;1,\dots,1)$. Together with the metric $g$ on the boundary this represents
the desired element $(W,U,g)$ of $R^{\spin}_{n}(BC_p)$.

\begin{remark}
In particular, \cite[Theorem 0.1]{BotvinnikGilkey_eta} ---the
  non-triviality of $R^{\spin}_{n}(BC_p)$--- extends to
  $n=4$ without any difficulty. Indeed, Botvinnik and Gilkey could have stated
  the result in this form and their proof would have worked. However, their
  main focus was on \cite[Theorem 0.2]{BotvinnikGilkey_eta} ---that the space
  of metrics of positive scalar curvature on a given manifold has infinitely many
  component. And this latter result we still can only prove for dimensions
  $\ge 5$. 
\end{remark}

\subsubsection{Case $n\equiv 2\pmod 4$}

Here, we have nothing new to offer, the construction of $\ind_\rho$ and $L$
are given in \cite{BotvinnikGilkey_eta}. One uses an appropriate virtual
representation $\rho$ of dimension $0$ and
  the associated twisted index (which then vanishes on the image of $\mu$) and
  a suitable explicitly constructed $L$. For the readers convenience we recall
  the construction. Start with the sphere bundle $S$ of the bundle
  $\eta\oplus \eta\to \complexs P^1$, where $\eta$ is the tautological complex
  line bundle over $\complexs P^1$. Use the diagonal $U(1)$-action to
  divide by the action of the cyclic subgroup $C_p$ to obtain a manifold
  $X$. Then set $L_0:= X\times K^r$ where $K$ is the Kummer surface, a
  $4$-dimensional simply connected spin manifold with non-vanishing $\hat
  A$-genus. In \cite[Lemma 2.3]{BotvinnikGilkey_eta} a spin structure, positive
  scalar curvature metric and classifying map to $BC_p$ are constructed for
  $L_0$. We can now continue exactly as in the case $n\equiv 0\pmod 4$, using
  $L_0$ instead of 
  $L(p,1,\dots,1)$ to obtain $(W,U,g)$ in the new case.
  In \cite{BotvinnikGilkey_eta}, the
  associated $\eta$-invariant homomorphism $\eta_\rho\colon
  \Pos^{\spin}_{n-1}(BG)\to \reals$ 
  is used and computed. It is not stated, but implicit in
  \cite{BotvinnikGilkey_eta} (and a 
  direct consequence of the Atiyah-Patodi-Singer index theorem) 
  that the composition $ R_{n}^{\spin}(BG)\to
  \Pos_{n-1}^{\spin}(BG)\xrightarrow{\eta_\rho}\reals$ equals
  $R_n^{\spin}(BG)\xrightarrow{\Ind} K_n(C^*
  G)\xrightarrow{\Ind_\rho}\reals$, so that the program described above goes
  through.

\begin{remark}
  The map $\ind_\rho$ out of $K_n(C^*G)$ is really just a tiny special
  case of the idea to map the diagram \eqref{eq:Stolz_seq} first to
  analysis/K-theory as in \eqref{eq:Gamma} but then further to non-commutative
  homology, to then pair with non-commutative cohomology, as carried out in
  \cite{PiazzaSchickZenobi} in general, and also with a different approach in
  \cite{XieYu, ChenWangXieYu}.
  Very similar constructions based on non-commutative cohomology and higher
  index theory, for groups constructed even more generally than
  $\Gamma=G\times \integers$, have been used in \cite{LeichtnamPiazza} to
  prove non-triviality of $\Pos^{\spin}_n(B\Gamma)$.
\end{remark}

%\begin{remark}
% Let $G$ be a group with elements of pairwise different prime order
% $p_1,\dots,p_k$.  
% The method of Weinberger and Yu allows to construct a
%free abelian subgroup of rank $k$ in  $R^{\spin}_{4k}(BG)$ if $k\ge 1$ is even
%and $G$ is finitely embeddable into a Hilbert space (a concept defined by Weinberger and Yu). This
%free abelian subgroup is mapped injectively under $\ind$ to
%$K_0(C^*G)/\im(\mu)$.
%
%Therefore in this situation our methods apply to get, just taking cartesian
%product with $S^1$, a free abelian subgroup of rank $k$ in
%$R^{\spin}_{4k+1}(BG\times\integers)$ mapping injectively to
%$\Pos^{\spin}_{4k}(BG\times \integers)$.
%
%\end{remark}

\begin{remark}
	If one assumes something more about $G$, it is possible to prove even finer results.
Let $G$ be a group with torsion. Under the assumption that for $G$ holds the Strong Novikov Conjecture, namely that the assembly map for $G$ is rationally injective, 
 \cite[Corollary 3.13]{XieYuZeidler} gives a good lower bound for the rank of $\Pos^{\spin}_{n}(B(G\times \integers))$ for any $n\geq3$ and not only for its cardinality.
 
 On the other hand,
let us consider the weaker assumption that $G$ is finitely embeddable into a Hilbert space (see \cite[Definition 1.3]{WeinbergerYu}).  If $G$ contains elements of pairwise different prime order
 $\{p_1,\dots,p_l\}$, then the method of Weinberger and Yu allows to construct a
 free abelian subgroup of rank $l$ in  $R^{\spin}_{4k}(BG)$ for $k\ge 1$. This
 free abelian subgroup is mapped injectively under $\ind$ to
 $K_0(C^*G)/\im(\mu)$.
 Therefore, in this situation our methods apply to get, just taking Cartesian
 product with $S^1$, a free abelian subgroup of rank $l$ in
 $R^{\spin}_{4k+1}(B(G\times\integers))$ mapping injectively to
 $\Pos^{\spin}_{4k}(B(G\times \integers))$.
 It is worth to mention that the relevant statements (and of course also the proof)
 in \cite{WeinbergerYu} are not completely correct. However, in \cite[Section
 3.2]{XieYuZeidler} the authors 
 address the gaps in the original proof of \cite[Corollary 4.2]{WeinbergerYu}
 for manifolds of dimension $4k+3$, which we are using here.
\end{remark}

\section{Positive scalar curvature bordism}

To complete the proof of Theorem \ref{theo:main}, one has to show that we can
actually find representatives which are connected and such that the map to
$B\Gamma$ induces an isomorphism $\pi_1(M)\to\Gamma$. However, this is a
general and well-known fact from the Gromov-Lawson surgery construction of
positive scalar 
curvature metrics and is already carried out in \cite[Section
5]{KazarasRubermanSaveliev} and \cite[Section 9.3]{MrowkaRubermanSaveliev}. As
we have nothing new to offer here, we don't repeat this proof.

\begin{remark}
  In dimension $n\ge 5$, one can do even better: given an arbitrary connected
  closed spin 
  manifold $M$ with reference map $f\colon M\to B\Gamma$ which is an isomorphism on
  $\pi_1$ and a manifold $f\colon M'\to B\Gamma$ with positive scalar
  curvature metric $g'$
  bordant to $M\to B\Gamma$, one can find a metric $g$ on $M$ of positive
  scalar curvature such that $[M',f',g']=[M,f,g]\in
  \Pos^{\spin}_n(B\Gamma)$. In other words: one can often choose the
  underlying manifold $M$ representing a class in
  $\Pos^{\spin}_n(B\Gamma)$.%  This is another main theme of much of the work on
  % positive scalar curvature: one gets information about the space of metrics
  % of positive scalar curvature on the manifold $M$.

  If $n=4$, however, this method breaks down. Nonetheless, some trace of this
  remains true, as shown in \cite[Section 9]{MrowkaRubermanSaveliev}: in the
  above situation, one can represent $[M',f',g']$ by $(\tilde M,\tilde f,g)$
  where $\tilde M$ is the connected sum of $M$ with an (unspecified) number of
  copies of $S^2\times S^2$.

  % As a consequence, the method does not give information about the space of
  % metrics of positive scalar curvature on $M$ itself, but at least about a ``stable''
  % version, where the stabilization is done by connected sum with $S^2\times
  % S^2$. In particular, as each of the infinitely many metrics is
  % supported on the connected sum with a finite number of copies of $S^2\times
  % S^2$, one finds $4$-manifolds with arbitrary lower bound on the number of
  % components of its space of metrics of positive scalar
  % curvature. 
\end{remark}


\begin{bibdiv}
  \begin{biblist}
\bib{Atiyah}{article}{
   author={Atiyah, M. F.},
   title={Elliptic operators, discrete groups and von Neumann algebras},
   conference={
      title={Colloque ``Analyse et Topologie'' en l'Honneur de Henri Cartan
      },
      address={Orsay},
      date={1974},
   },
   book={
      publisher={Soc. Math. France, Paris},
   },
   date={1976},
   pages={43--72. Ast\'{e}risque, No. 32-33},
   review={\MR{0420729}},
}%     \bib{BotvinnikGilkey_spaceforms}{article}{
%    author={Botvinnik, Boris},
%    author={Gilkey, Peter B.},
%    title={Metrics of positive scalar curvature on spherical space forms},
%    journal={Canad. J. Math.},
%    volume={48},
%    date={1996},
%    number={1},
%    pages={64--80},
%    issn={0008-414X},
%    review={\MR{1382476}},
%    doi={10.4153/CJM-1996-003-0},
% }
		

\bib{BotvinnikGilkey_eta}{article}{
   author={Botvinnik, Boris},
   author={Gilkey, Peter B.},
   title={The eta invariant and metrics of positive scalar curvature},
   journal={Math. Ann.},
   volume={302},
   date={1995},
   number={3},
   pages={507--517},
   issn={0025-5831},
   review={\MR{1339924}},
   doi={10.1007/BF01444505},
 }
 
 		\bib{ChenWangXieYu}{unpublished}{
 			author={Chen,  Xiaoman},
 			author={Wang, Jinmin},
 			author={Xie, Zhizhang},
 			author={Yu, Guoliang},
 			title={Delocalized eta invariants, cyclic cohomology and higher rho invariants}, 
 			note={\href{http://www.arXiv.org/abs/1901.02378}{arXiv:1901.02378}},
 			date={2019}
 		}
 
 \bib{Donnelly}{article}{
   author={Donnelly, Harold},
   title={Eta invariants for $G$-spaces},
   journal={Indiana Univ. Math. J.},
   volume={27},
   date={1978},
   number={6},
   pages={889--918},
   issn={0022-2518},
   review={\MR{511246}},
   doi={10.1512/iumj.1978.27.27060},
}
 \bib{Gilkey}{article}{
   author={Gilkey, P. B.},
   title={The eta invariant and the $K$-theory of odd-dimensional spherical
   space forms},
   journal={Invent. Math.},
   volume={76},
   date={1984},
   number={3},
   pages={421--453},
   issn={0020-9910},
   review={\MR{746537}},
   doi={10.1007/BF01388468},
}
		
\bib{HansonRoemer}{article}{
  author={Hanson, Andrew J.},
  author={R\"omer, Hartmann},
  title={Gravitational instanton contribution to Spin 3/2 axial anomaly},
  journal={Phys. Lett.},
  volume={80B},
  date={1978},
  pages={58--60},
  }

    \bib{LeichtnamPiazza}{article}{
   author={Leichtnam, Eric},
   author={Piazza, Paolo},
   title={On higher eta-invariants and metrics of positive scalar curvature},
   journal={$K$-Theory},
   volume={24},
   date={2001},
   number={4},
   pages={341--359},
   issn={0920-3036},
   review={\MR{1885126}},
   doi={10.1023/A:1014079307698},
 }
 \bib{MrowkaRubermanSaveliev}{article}{
   author={Mrowka, Tomasz},
   author={Ruberman, Daniel},
   author={Saveliev, Nikolai},
   title={An index theorem for end-periodic operators},
   journal={Compos. Math.},
   volume={152},
   date={2016},
   number={2},
   pages={399--444},
   issn={0010-437X},
   review={\MR{3462557}},
   doi={10.1112/S0010437X15007502},
}
 \bib{KazarasRubermanSaveliev}{unpublished}{
   author={Kazaras,  Demetre},
   author={Ruberman, Daniel},
   author={Saveliev, Nikolai},
   title={On positive scalar curvature cobordisms and the conformal Laplacian
     on end-periodic manifolds},
  note={\href{http://www.arXiv.org/1902.00443}{arXiv:1902.00443}},
  note={to appear in Communications in Analysis and Geometry},
  date={2019},
}
\bib{PiazzaSchick}{article}{
   author={Piazza, Paolo},
   author={Schick, Thomas},
   title={Groups with torsion, bordism and rho invariants},
   journal={Pacific J. Math.},
   volume={232},
   date={2007},
   number={2},
   pages={355--378},
   issn={0030-8730},
   review={\MR{2366359}},
   doi={10.2140/pjm.2007.232.355},
}
\bib{PiazzaSchick_Stolz}{article}{
   author={Piazza, Paolo},
   author={Schick, Thomas},
   title={Rho-classes, index theory and Stolz' positive scalar curvature
   sequence},
   journal={J. Topol.},
   volume={7},
   date={2014},
   number={4},
   pages={965--1004},
   issn={1753-8416},
   review={\MR{3286895}},
   doi={10.1112/jtopol/jtt048},
 }
 \bib{PiazzaSchickZenobi}{unpublished}{
   author={Piazza, Paolo},
   author={Schick, Thomas},
   author={Zenobi, Vito},
   title={Mapping Analytic Surgery to Homology, Higher Rho Numbers and Metrics
     of Positive Scalar Curvature},
   eprint={arXiv:1905.11861v5},
   note={{\tt arXiv:1905.11861v5}, to appear in Memoirs of the Amer. Math. Soc.},
   date={2019},
 }
 \bib{Ramachandran}{article}{
   author={Ramachandran, Mohan},
   title={von Neumann index theorems for manifolds with boundary},
   journal={J. Differential Geom.},
   volume={38},
   date={1993},
   number={2},
   pages={315--349},
   issn={0022-040X},
   review={\MR{1237487}},
 }
  \bib{SchickSeyedhosseini}{article}{
   author={Schick, Thomas},
   author={Seyedhosseini, Mehran},
   eprint={arXiv:1811.08142},
   journal={M\"unster J. Math.},
   number={14},
   number={2},
   pages={123--154},
   eprint={arXiv:1811.08142},
   title={On an Index Theorem of Chang, Weinberger and Yu},
   date={2021},
   }
\bib{Schochet}{article}{
   author={Schochet, Claude},
   title={Topological methods for $C^{\ast} $-algebras. II. Geometric
   resolutions and the K\"{u}nneth formula},
   journal={Pacific J. Math.},
   volume={98},
   date={1982},
   number={2},
   pages={443--458},
   issn={0030-8730},
   review={\MR{650021}},
}
	
\bib{WeinbergerYu}{article}{
	author={Weinberger, Shmuel},
	author={Yu, Guoliang},
	title={Finite part of operator $K$-theory for groups finitely embeddable
		into Hilbert space and the degree of nonrigidity of manifolds},
	journal={Geom. Topol.},
	volume={19},
	date={2015},
	number={5},
	pages={2767--2799},
	issn={1465-3060},
	review={\MR{3416114}},
	doi={10.2140/gt.2015.19.2767},
}
 \bib{XieYu}{article}{
 author={Xie, Zhizhang},
  author={Yu, Guoliang},
title={Delocalized eta invariants, algebraicity, and K-theory of group
  $C^*$-algebras},
journal={Int. Math. Res. Not.},
number={15},
pages={11731--11766},
date={2021},
eprint={arXiv:1805.07617}
}		

 \bib{XieYuZeidler}{article}{
 author={Xie, Zhizhang},
 author={Yu, Guoliang},
 	author={Zeidler, Rudolf},
 	title={
On the range of the relative higher index and the higher rho-invariant for
positive scalar curvature},
journal={Adv. Math.},
note={Paper No. 107897, 24pp.},
eprint={arXiv:1712.03722v2},
 	date={2021},
 }


    % amsbib entries; as from mathscinet

		
  \end{biblist}
\end{bibdiv}
\end{document}